\documentclass[12pt]{amsart}

\usepackage{amsmath,mathtools,amsthm,amsfonts,amssymb,color,verbatim,graphicx,cases}
\usepackage[letterpaper]{geometry}
\usepackage{stmaryrd}

\theoremstyle{definition}
\newtheorem{theorem}[equation]{Theorem}
\newtheorem{corollary}[equation]{Corollary}

\newtheorem{lemma}[equation]{Lemma}
\newtheorem{proposition}[equation]{Proposition}

\theoremstyle{definition}
\newtheorem{definition}[equation]{Definition}

\newtheorem{exam}[equation]{Example}
\numberwithin{equation}{section}

\definecolor{mjo}{rgb}{0,0,.9}

\global\long\def\gen{\text{\sf GEN}}
\global\long\def\dng{\text{\sf DNG}}
\global\long\def\mex{\operatorname{mex}}
\global\long\def\nim{\operatorname{nim}}
\global\long\def\opt{\operatorname{Opt}}

\global\long\def\pty{\operatorname{pty}}
\global\long\def\type{\operatorname{type}}
\global\long\def\otype{\operatorname{otype}}

\global\long\def\spr{d}

\global\long\def\dih{\operatorname{Dih}}
\global\long\def\dcy{\delta}

\begin{document}
\setlength{\jot}{0pt} 
\title
[Impartial achievement games for generating generalized dihedral groups]
{Impartial achievement games for generating \\
generalized dihedral groups}

\author{Bret J.~Benesh}
\author{Dana C.~Ernst}
\author{N\'andor Sieben}
\subjclass[2010]{91A46, 
20D30
}

\keywords{impartial game, maximal subgroup, generalized dihedral group}

\date{\today}

\maketitle


\begin{abstract}
We study an impartial game introduced by Anderson and Harary.  This game is played by two players who alternately choose previously-unselected elements of a finite group. The first player who builds a generating set from the jointly-selected elements wins. We determine the nim-numbers of this game for generalized dihedral groups, which are of the form $\operatorname{Dih}(A)= \mathbb{Z}_2 \ltimes A$ for a finite abelian group $A$.
\end{abstract}


\section{Introduction}


Anderson and Harary~\cite{anderson.harary:achievement} introduced a pair of two-player impartial games on a nontrivial finite group $G$.  In both games, two players take turns choosing a previously-unselected element of $G$. The player who first builds a generating set from the jointly-selected elements wins the achievement game $\gen(G)$, while a player who cannot avoid building a generating set loses the avoidance game $\dng(G)$.  
The theory of these games is further developed in \cite{Barnes,BeneshErnstSiebenSymAlt,BeneshErnstSiebenDNG,ErnstSieben}. Similar algebraic games are studied in \cite{brandenburg:algebraicGames}.

For a finite abelian group $A$, the \emph{generalized dihedral group} $\dih(A)$ is the semidirect product $\mathbb{Z}_{2}\ltimes A$.
Generalized dihedral groups share many of the same properties of dihedral groups.  For example, every element of $\dih(A)$ that is not in $A$ has order $2$ and acts on $A$ by inversion.  See~\cite{IsaacsFiniteGroups} or~\cite{MR0167513} for an introduction to generalized dihedral groups.

A fundamental problem in game theory is to determine nim-numbers of impartial two-player games. 
The nim-number allows for the easy calculation of the outcome of the sum of games. A general theory on how to compute nim-numbers appears in~\cite{ErnstSieben}.  
The strategies for the avoidance game for many families of finite groups were presented in~\cite{Barnes}, and a complete theory for finding the nim-numbers for the avoidance games was developed in~\cite{BeneshErnstSiebenDNG}.   
The strategies and nim-numbers for symmetric and alternating groups were determined in~\cite{BeneshErnstSiebenSymAlt} for both the achievement and avoidance games. 

The nim-numbers for $\dng(\dih(A))$ were classified in~\cite{BeneshErnstSiebenDNG}, which found that $\dng(\dih(A))$ is $*3$ if $A$ is cyclic of odd order and $*0$ otherwise.  
The task in this paper is to determine the nim-numbers for $\gen(\dih(A))$ for all finite abelian $A$.


\begin{section}{Preliminaries}


In this section, we recall some general terminology and results from~\cite{BeneshErnstSiebenSymAlt,BeneshErnstSiebenDNG, ErnstSieben} related to impartial games, as well as the achievement game that is the focus of this paper.


\subsection{Impartial Games}


A comprehensive treatment of impartial games can be
found in~\cite{albert2007lessons,SiegelBook}. An \emph{impartial
game} is a finite set $X$ of \emph{positions} together with a collection $\{\opt(P)\subseteq X\mid P\in X\}$, where $\opt(P)$ is the set of
possible \emph{options} for a position $P$. Two players take turns replacing the current position $P$ with one of the available options in $\opt(P)$. 
The player who encounters a \emph{terminal position} with
an empty option set cannot move and therefore \emph{loses}. All games
must come to an end in finitely many turns, so we do not allow infinite
lines of play. 

The one-pile NIM game $*n$ with $n$ stones is the prototype of an impartial game. 
The set of options of $*n$ is $\opt(*n)=\{*0,\ldots,*(n-1)\}$.

The \emph{minimum excludant} $\mex(A)$ of a set $A$ of ordinals
is the smallest ordinal not contained in the set. The \emph{nim-number}
$\nim(P)$ of a position $P$ is the minimum excludant of the set
of nim-numbers of the options of $P$. That is, 
\[
\nim(P):=\mex\{\nim(Q)\mid Q\in\opt(P)\}.
\]
The minimum excludant of the empty set is $0$, so the terminal positions
of a game have nim-number $0$. 
The \emph{nim-number of a game} is the nim-number of its starting position. 
A winning strategy exists at position $P$ if and only if $\nim(P)\ne 0$.
So, the nim-number of a game determines the outcome of the game. 
We write $P=R$ if $\nim(P)=\nim(R)$, so $P=*\nim(P)$ for every impartial game $P$.


\subsection{Achievement Games for Groups}


We now give a more precise description of the achievement game $\gen(G)$ played on a finite group $G$. We also recall some definitions and results from~\cite{ErnstSieben}. 
The nonterminal positions of $\gen(G)$ are exactly the non-generating subsets of $G$, 
and the terminal positions are the generating sets $S$ of $G$ such that there is an $g \in S$ satisfying $\langle S \setminus \{g\}\rangle < G$.  
The starting position is the empty set since neither player has chosen an element yet.
The first player chooses $x_{1}\in G$, and the designated player selects $x_{k}\in G\setminus\{x_{1},\ldots,x_{k-1}\}$ at the $k$th turn.
A position $Q$ is an option of $P$ if $Q=P \cup \{g\}$ for some $g \in G \setminus P$.  The player who builds a generating set wins the game. 

The set $\mathcal{M}$ of maximal subgroups play a significant role in the game.  The last two authors~\cite{ErnstSieben} define the set \[ \mathcal{I}:=\{\cap\mathcal{N}\mid\emptyset\not=\mathcal{N\subseteq\mathcal{M}}\} \] of \emph{intersection subgroups}, which is the set of all possible intersections of maximal subgroups.  The smallest intersection subgroup is the Frattini subgroup $\Phi(G)$ of $G$, which is the intersection of all maximal subgroups of $G$. 

\begin{exam}
\label{running}
Let $G=\langle r,s \mid r^4=e=s^2,sr^3=rs\rangle$ be the dihedral group of order 8 with identity $e$. The maximal subgroups $\langle r \rangle$, $\langle s,sr^2 \rangle$, and $\langle sr,sr^3 \rangle$ are the order 4 intersection subgroups. The only other intersection subgroup is the Frattini subgroup $\langle r^2 \rangle$, which has order 2.
\end{exam}

The set $\mathcal{I}$ of intersection subgroups is partially
ordered by inclusion. We use interval notation to denote certain subsets
of $\mathcal{I}$. For example, if $I\in\mathcal{I}$, then $(-\infty,I):=\{J\in\mathcal{I}\mid J \lneq I\}$. 

For each $I\in\mathcal{I}$ let
\[
X_{I}:=\mathcal{P}(I)\setminus\cup\{\mathcal{P}(J)\mid J\in(-\infty,I)\}
\]
be the collection of those subsets of $I$ that are not contained
in any other intersection subgroup properly contained in $I$. We let $\mathcal{X}:=\{X_{I}\mid I\in\mathcal{I}\}$
and call an element of $\mathcal{X}$ a \emph{structure class}.  
We define an additional structure class $X_G$ to be the set of terminal positions, and we let $\mathcal{Y}:=\mathcal{X} \cup \{X_G\}$.   
For any position $P$ of $\gen(G)$, let $\lceil P \rceil$ be the unique element of $\mathcal{I}\cup \{G\}$ such that $P\in X_{\lceil P \rceil}$.  
Note that $\lceil P \rceil$ is the smallest intersection subgroup containing $P$ if $P$ is not a terminal position.  We will write $\lceil A,g_1,\ldots,g_n\rceil$ to mean $\lceil A \cup \{g_1,\ldots,g_n\}\rceil$ for $A \subseteq G$ and $g_1,\ldots,g_n \in G$.

\begin{exam}
We continue Example~\ref{running}. The structure classes in $\mathcal{X}$ are 
\begin{align*}
X_{\langle r^2 \rangle} &= \{\emptyset, \{e\}, \{r^2\}, \{e, r^2\} \}\\ 
X_{\langle r \rangle} &= \mathcal{P}(\langle r \rangle)\setminus X_{\langle r^2 \rangle} \\
X_{\langle s,sr^2 \rangle} &= \mathcal{P}(\langle s,sr^2 \rangle)\setminus X_{\langle r^2 \rangle}\\ 
X_{\langle sr,sr^3 \rangle} &= \mathcal{P}(\langle sr,sr^3 \rangle)\setminus X_{\langle r^2 \rangle}.
\end{align*}
Thus, for example $\lceil \emptyset  \rceil=\langle r^2 \rangle=\lceil e \rceil$, $\lceil r \rceil = \langle r \rangle$, and $\lceil s \rceil=\langle s ,sr^2\rangle = \lceil s,sr^2 \rceil$.  
\end{exam}

Parity plays a crucial role in the theory of impartial games.
We define the \emph{parity of a natural number $n$} via $\pty(n):=(1+(-1)^{n+1})/2$. The \emph{parity of a subset of a group} is defined to be the parity of the size of the subset.
 The \emph{parity of a structure class} is defined to be $\pty(X_{I}):=\pty(I)$.  We will say an object to be \emph{even} if its parity is $0$ and \emph{odd} if its parity is $1$. 

The partition $\mathcal{Y}$ 
of the set of game positions of $\gen(G)$ is compatible with
the option relationship between game positions~\cite[Corollary~3.11]{ErnstSieben}: if $X_{I},X_{J}\in\mathcal{X}$
and $P,Q\in X_{I}\ne X_{J}$, then $\opt(P)\cap X_{J}\not=\emptyset$
if and only if $\opt(Q)\cap X_{J}\ne\emptyset$.  We say that $X_{J}$ is an \emph{option} of $X_{I}$ and we write
$X_{J}\in\opt(X_{I})$ if $\opt(I)\cap X_{J}\not=\emptyset$. 

If $P,Q\in X_{I}\in \mathcal{Y}$ and $\pty(P)=\pty(Q)$, then $\nim(P)=\nim(Q)$ by~\cite[Proposition~4.4]{ErnstSieben}.  The \emph{structure digraph} of $\gen(G)$ has vertex set $\mathcal{Y}$ and edge set $\{(X_{I},X_{J})\mid X_{J}\in\opt(X_{I})\}$.  In a \emph{structure diagram,} a structure class $X_{I}$ is represented
by a triangle pointing down if $I$ is odd and by a triangle pointing
up if $I$ is even. The triangles are divided into a smaller triangle
and a trapezoid, where the smaller triangle represents the odd positions
of $X_{I}$ and the trapezoid represents the even positions of $X_{I}$.
The numbers in the smaller triangle and the trapezoid are the nim-numbers
of these positions. There is a directed arrow from $X_{I}$ to $X_{J}$
provided $X_{J}\in\opt(X_{I})$.

\begin{figure}
\includegraphics{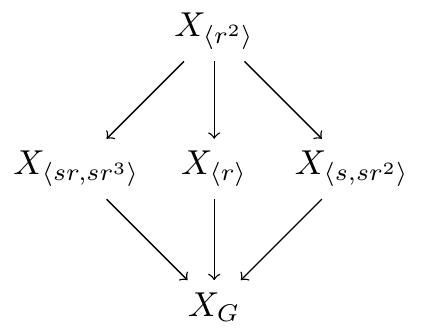}
$\qquad$
\includegraphics{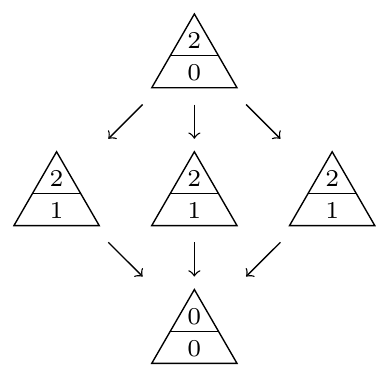}
$\qquad$
\includegraphics{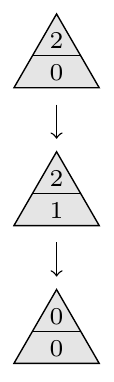}
\caption{\label{D8unsimp}A depiction of the structure digraph, the structure diagram, and the simplified structure diagram, respectively, of $\gen(G)$ for the dihedral group $G$ of order 8.}
\end{figure}

\begin{exam}
Figure~\ref{D8unsimp} shows the structure digraph and the structure diagram of $\gen(G)$ for the dihedral group $G$ of Example~\ref{running}. The numbers in the diagram were computed from the bottom up using nim-number arithmetic as depicted in Figure~\ref{fig:type}.
\end{exam}

The \emph{type} of the structure class $X_{I}$ is the triple 
\[
\type(X_{I}):=(\pty(I),\nim(P),\nim(Q)),
\]
where $P,Q\in X_{I}$ with $\pty(P)=0$ and $\pty(Q)=1$. Note that the type of a structure class $X_I$ is determined by the parity of $X_I$ and the types of the options of $X_I$ as shown in Figure~\ref{fig:type}.  
We define 
$\otype(X_I):=\{\type(X_J) \mid X_J \in \opt(X_I)\}$.

\begin{figure}
\includegraphics{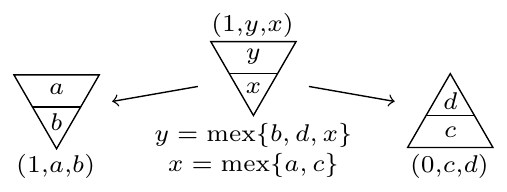} \hfil \includegraphics{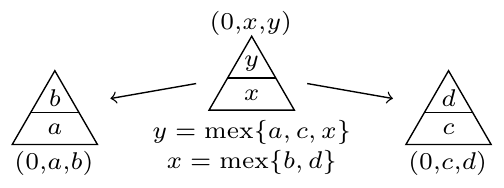}
\caption{\label{fig:type}Examples for the calculation of the type of a structure class using the types of the options.
}
\end{figure}

The nim-number of the game  is the  nim-number 
of the initial position $\emptyset$, which is an even subset of $\Phi(G)$. 
Because of this, $\nim(\gen(G))$ is the second component of 
\[
\type(X_{\Phi(G)}) =(\pty(\Phi(G)),\nim(\emptyset),\nim(\{e\})),
\] 
which corresponds to the trapezoidal 
part of the triangle representing the source vertex $X_{\Phi(G)}$ of the structure diagram.

The \emph{simplified structure diagram} of $\gen(G)$ is built from the structure diagram by identifying 
structure classes 
$X_I$ and $X_J$ satisfying $\type(X_I)=\type(X_J)$ and $\otype(X_I) \cup \{\type(X_I)\} = \otype(X_J) \cup \{\type(X_J)\}$.
This is followed by the removal of any resulting loops. See~\cite{ErnstSieben} for a more detailed description.

\begin{exam}
The simplified structure diagram of $\gen(G)$ is shown in Figure~\ref{D8unsimp} for the dihedral group $G$ of Example~\ref{running}. 
\end{exam}

\end{section}


\section{Deficiency}


We will develop some general tools in this section.  For a finite group $G$, the smallest possible size of a generating set is denoted by \[\spr(G) :=  \min\{|S| : \langle S \rangle=G\}.\]   The following definition is closely related to $\spr(G)$. 

\begin{definition}
The \emph{deficiency} of a subset $P$ of a finite group $G$ is the minimum size $\dcy(P)$ of a subset $Q$ of $G$ such that $\langle P\cup Q\rangle=G$.   
\end{definition}

Note that the deficiency of a generating set is $0$ and $\delta(\emptyset)=d(G)$.  Also, $P \subseteq Q$ implies $\dcy(P) \geq \dcy(Q)$.

\begin{proposition}\label{prop:EqualDeficiencies}
If $P,Q\in X_I$, then $\dcy(P)=\dcy(Q)$. 
\end{proposition}

\begin{proof}
Let $\mathcal{L}$ be the set of maximal subgroups containing $I$.  Then $I=\cap \mathcal{L}$ because $I$ is an intersection subgroup.  Let $g_1,\ldots,g_{\dcy(P)} \in G$ such that $\langle P,g_1,\ldots,g_{\dcy(P)} \rangle = G$.  Then no element of $\mathcal{L}$ contains $\{g_1,\ldots,g_{\dcy(P)}\}$.  The only maximal subgroups that contain $Q$ are also in $\mathcal{L}$. Therefore, $\langle Q,g_1,\ldots,g_{\dcy(P)}\rangle=G$ and $\dcy(Q) \leq \dcy(P)$.  By a symmetric argument, $\dcy(P) \leq \dcy(Q)$, so $\dcy(P)=\dcy(Q)$. 
\end{proof}

\begin{corollary}\label{cor:PhiDeficiency}
If $G$ is a finite group, then $\dcy(\Phi(G)) = d(G)$.
\end{corollary}

\begin{proof}
Since $\emptyset,\Phi(G)\in X_{\Phi(G)}$, we have $\dcy(\Phi(G))=\dcy(\emptyset)=d(G)$.
\end{proof}

\begin{definition}
We say that a structure class $X_I\in\mathcal{Y}$ is {\it $m$-deficient in} $G$ and we write $\dcy(X_I)=m$ if $\dcy(I)=m$. 
\end{definition}

The terminal structure class is $0$-deficient. That is, $\dcy(X_G)=0$.

\begin{definition}
Let $G$ be a finite group, $\mathcal{E}$ be the set of even structure classes and $\mathcal{O}$ be the set of odd structure classes in $\mathcal{Y}$.
We define the following sets:
\[
\begin{aligned}
\mathcal{D}_m &:= \{ X_I\in \mathcal{Y} \mid \dcy(X_I)=m \}, & \mathcal{D}_{\geq m} &:= \cup_{k \geq m} \mathcal{D}_k \\ 
\mathcal{E}_m &:= \mathcal{E}\cap \mathcal{D}_m, & \mathcal{E}_{\geq m} &:= \cup_{k \geq m} \mathcal{E}_k \\ 
\mathcal{O}_m &:= \mathcal{O}\cap \mathcal{D}_m,  & \mathcal{O}_{\geq m} & := \cup_{k \geq m} \mathcal{O}_k \\
\end{aligned}
\]
\end{definition}

Note that $\mathcal{D}_0=\{X_G\}$.

\begin{proposition}\label{prop:GENemptysets}
If $G$ is a finite group, then $\mathcal{D}_{\geq d(G)+1}=\emptyset$.
\end{proposition}
\begin{proof}
There must be a generating set $Q$ for $G$ with size $d(G)$. If $X_I \in \mathcal{D}_m$, then $G=\langle Q \rangle \subseteq \langle I\cup Q \rangle$, and so $m \leq |Q| = d(G)$.   
\end{proof}

The next proposition follows immediately from Corollary~\ref{cor:PhiDeficiency}.

\begin{proposition}\label{prop:frattinilocation}
If $G$ is a finite group, then $X_{\Phi(G)} \in \mathcal{D}_{d(G)}$.
\end{proposition}

\begin{proposition}\label{prop:setoptions}
Let $G$ be a finite group and $m$ be a positive integer. If $X_I \in \mathcal{D}_m$, 
then $X_I$ has an option $X_K$ in $\mathcal{D}_{m-1}$, and every option of $X_I$ is in either $\mathcal{D}_m$ or $\mathcal{D}_{m-1}$.
\end{proposition}
\begin{proof}
Let $X_J$ be an option of $X_I$; 
that is, $Q:=I\cup\{g\}$ is in $X_J$ for some $g\in G$. Since $I\subseteq Q$, \[\delta(X_I)=\delta(I)\ge\delta(Q)=\delta(J)=\delta(X_J).\]  Let $h_1,\ldots,h_{\delta(Q)}\in G$ such that $\langle Q,h_1,\ldots,h_{\delta(Q)}\rangle=G$.  Then $\langle I,g,h_1,\ldots,h_{\delta(Q)}\rangle=G$, which implies \[\delta(X_I)=\delta(I)\le\delta(Q)+1=\delta(X_J)+1.\]  Hence $\dcy(X_J) \in \{\dcy(X_I),\dcy(X_I)-1\}$.

Now let $g_1,\ldots,g_{\delta(I)}\in G$ such that $\langle I,g_1,\ldots,g_{\delta(I)}\rangle=G$, and let $K=\lceil I \cup \{g_1\} \rceil$. Then $\dcy(X_K) \in \{\dcy(X_I),\dcy(X_I)-1\}$.  Because $\langle I \cup \{g_1\},g_2,\ldots,g_{\delta(I)}\rangle = G$, we conclude that $\dcy(X_K)=\dcy(X_I)-1$.
\end{proof}

Note that the deficiency of a structure class $X_I$ is the directed distance from $X_I$ to the terminal structure class $X_G$ in the structure digraph.

Even structure classes only have even options by Lagrange's Theorem, so the next corollary follows by Proposition \ref{prop:setoptions}.

\begin{corollary}\label{cor:Esetoptions}
If $m$ is a positive integer and $X_I \in \mathcal{E}_m$, then $X_I$ has an option in $\mathcal{E}_{m-1}$ and every option of $X_I$ is in either $\mathcal{E}_m$ or $\mathcal{E}_{m-1}$.
\end{corollary}

The proof of the following result uses calculations similar to those shown in Figure~\ref{fig:type}.

\begin{proposition}\label{prop:GENeventypes}
If $G$ is a group of even order, then 
\[
	   \type(X_I) = 
	     \begin{dcases}
	       (0,0,0) & \text{if } X_I=X_G \in \mathcal{E}_0\\
	       (0,1,2) & \text{if } X_I \in \mathcal{E}_1\\
	       (0,0,2) & \text{if } X_I \in \mathcal{E}_2\\
	       (0,0,1) & \text{if } X_I \in \mathcal{E}_{\geq 3}.
	     \end{dcases}
\]
\end{proposition}

\begin{proof}
We will use structural induction on the structure classes.
By Corollary~\ref{cor:Esetoptions}, structure classes in $\mathcal{E}_m$ with $m \geq 1$ must have an option in $\mathcal{E}_{m-1}$ and possibly more options 
in $\mathcal{E}_{m}\cup\mathcal{E}_{m-1}$.
Hence $\mathcal{E}_m\ne\emptyset$ if and only if $m\le \spr(G)$.

If $X_I \in \mathcal{E}_0$, then $\type(X_I)=\type(X_G)=(0,0,0)$.
If $X_I \in \mathcal{E}_1$, then $\type(X_I)=(0,1,2)$ since
\[
\otype(X_I)=
\begin{dcases}
\{(0,0,0)\} & \text{if } \opt(X_I)\cap\mathcal{E}_1=\emptyset\\
\{(0,0,0),(0,1,2)\} & \text{otherwise.} 
\end{dcases}
\]
If $X_I \in \mathcal{E}_2$, then $\type(X_I)=(0,0,2)$ since
\[
\otype(X_I)=
\begin{dcases}
\{(0,1,2)\} & \text{if }  \opt(X_I)\cap\mathcal{E}_2=\emptyset\\
\{(0,1,2),(0,0,2)\} & \text{otherwise.} 
\end{dcases}
\]
If $X_I  \in \mathcal{E}_{3}$, then $\type(X_I)=(0,0,1)$ since
\[
\otype(X_I)=
\begin{dcases}
\{(0,0,2)\} & \text{if } \opt(X_I)\cap\mathcal{E}_3=\emptyset\\
\{(0,0,2),(0,0,1)\} & \text{otherwise.} 
\end{dcases}
\]
If $X_I \in \mathcal{E}_{\geq 4}$, then $\type(X_I)=(0,0,1)$, since every option of $X_I$ has type $(0,0,1)$.
\end{proof}

\begin{corollary}\label{cor:dgt4}
If $G$ is a group of even order with $\spr(G) \geq 4$, then $\gen(G)=*0$.
\end{corollary}

\begin{proof}
If the first player initially selects an involution, the second player selects the identity; otherwise, the second player selects an involution $t$.  This guarantees that the game position $P$ after the second turn satisfies 
$P\in X_I\in\mathcal{E}_{\ge d(G)-2}$ for some structure class $X_I$. Because $d(G)-2 \ge 2$, $\type(X_I)$ is $(0,0,2)$ or $(0,0,1)$ by Proposition~\ref{prop:GENeventypes}.
In both cases the even position $P$ has nim-number $0$, so the second player wins.
\end{proof}


\section{The Generalized Dihedral Group $\dih(A)$}


For a finite abelian group $A$, the \emph{generalized dihedral group} $\dih(A)$
is the semidirect product $C_{2}\ltimes A$ where the generator $x$ of the two-element cyclic group $C_2=\{1,x\}$
acts on $A$ on the right by inversion. 
That is, the operations on the elements of $C_2\times A$ are defined by 
\[
(k,a)(1,b):=(k,ab),\quad (k,a)(x,b):=(kx,a^{-1}b)
\] 
for $a,b \in A$ and $k\in C_2$.  

Note that we use multiplicative notation for group operations. 
We identify the elements of $A$ and $C_2$ with their images in $C_{2}\ltimes A$ through the natural inclusion maps $k\mapsto (k,1):C_2\hookrightarrow \dih(A)$ and $a\mapsto (1,a):A\hookrightarrow \dih(A)$ 
so that we can write $ka$ for $(k,a)$.  
This identification allows us to consider the action of $C_2$ on $A$ as conjugation since $(ka)(xb) = kxx^{-1}axb=(kx)(a^xb)$.

\begin{lemma}\label{lem:twoinvolutionstransform}
Let $S$ be a subset of $\dih(A)$ and $y \in \dih(A) \setminus A$.
If $a, b \in A$, then $\langle S, ya, yb \rangle = \langle S, a^{-1}b, yb \rangle$.  
\end{lemma}

\begin{proof}
Let $H=\langle S, ya, yb \rangle$ and $K=\langle S, a^{-1}b, yb \rangle$.  Then $S \subseteq H$, $yb \in H$, and $a^{-1}b = (yay)b = (ya)(yb) \in H$, so $K \leq H$.  Also, $S \subseteq K$, $yb \in K,$ and $ya=(yb)(b^{-1}a)=(yb)(a^{-1}b)^{-1} \in K$, so $H \leq K$.  Therefore, $H=K$.
\end{proof}

\begin{lemma}\label{lem:DihASubgroups}
If $\langle B,y \rangle=\dih(A)$ for some $B\le A$ and $y \in \dih(A) \setminus A$, then $A=B$.
\end{lemma}

\begin{proof}
It is easily verified that $y^{-1}by=b^{-1}$ for all $b\in B$. So $\langle y \rangle$ normalizes $B$, and hence $B$ is normal in $\langle B,y\rangle = \dih(A)$.  Thus, $\dih(A) = \langle B,y \rangle
\subseteq B\langle y \rangle$, which implies $\dih(A)=B\langle y \rangle$. Since $B\cap \langle y \rangle$ is trivial, $B\langle y \rangle=\langle y \rangle \ltimes B$ after the natural identification.  Since $y$ acts on $B$ by inversion, $\langle y \rangle \ltimes B = \dih(B)$. Therefore, $\dih(A)=\dih(B)$, which shows that $A=B$. 
\end{proof}

\begin{proposition}\label{prop:NumberOfGenerators}
If $A$ is a finite abelian group, then $d(\dih(A))=d(A)+1$.
\end{proposition}
\begin{proof}
Let $a_1,\ldots,a_{d(A)} \in A$ such that $\langle a_1,\ldots,a_{d(A)}\rangle =A$.  Since $|\dih(A):A|=2$, \[\langle a_1,\ldots,a_{d(A)},y\rangle=\dih(A)\] for all $y \in \dih(A) \setminus A$, so $d(\dih(A)) \leq d(A)+1$.  

Now let $n=d(\dih(A))$ and  $g_1,\ldots,g_n \in \dih(A)$ such that $\langle g_1,\ldots,g_n \rangle = \dih(A)$.  Without loss of generality, let $g_n \in \dih(A) \setminus A$.  
By possible repeated use of Lemma~\ref{lem:twoinvolutionstransform}
we may find $B:=\langle b_1,\ldots,b_{n-1}\rangle \leq A$ such that $\langle b_1,\ldots,b_{n-1},g_n \rangle = \langle g_1,\ldots,g_n\rangle.$    
Then $\langle B,g_n \rangle=\dih(A)$, so $B=A$ by Lemma~\ref{lem:DihASubgroups}.  Therefore, $d(A) \leq n-1 = d(\dih(A))-1$.   
\end{proof}

The following result is likely well-known, but we will include the proof for want of a reference.

\begin{proposition}\label{prop:FrattiniInclusion}
If $G=\dih(A)$, then $\Phi(G)=\Phi(A)$.  
\end{proposition}

\begin{proof}
Since $A$ is a maximal subgroup of $G$, $\Phi(G)$ is contained in $A$.  

We will first prove that $\Phi(A)$ is contained in $\Phi(G)$.  Let $z$ be in $\Phi(A)$ and $M$ be a maximal subgroup of $G$.  
If $M=A$, then $z$ is certainly in $M$, so we will assume that $M$ is not equal to $A$.  
Then $M \cap A$ is maximal in $A$, so $z \in M \cap A \subseteq M$.  Therefore, $\Phi(A) \subseteq \Phi(G)$.  

We now prove that $\Phi(G)$ is contained in $\Phi(A)$.  Let $z$ be in $\Phi(G)$ and $L$ be a maximal subgroup of $A$.  
Then $M= \langle x \rangle \ltimes L$ 
is a maximal subgroup of $G$, so $z \in M$.   
Since $z$ is in $\Phi(G)$ and $A$ is a maximal subgroup of $G$, $z$ is in $A$.
Therefore, $z \in M \cap A = L$.  
Therefore, $\Phi(G) \leq \Phi(A)$, and we conclude that $\Phi(G)=\Phi(A)$.    
\end{proof}


\section{The Achievement Game $\gen(\dih(A))$}


We now determine the nim-number of $\gen(\dih(A))$. We consider three cases: $d(A)=1$, $d(A)\geq 3$, and $d(A)=2$.  If $d(A)=1$, then $A$ is a cyclic group and $\dih(A)$ is a usual dihedral group characterized by the following result.

\begin{proposition}\cite[Corollary~7.10]{ErnstSieben}\label{prop:DihedralResults}
If $A$ is a cyclic group of order $n$, then 
\[
\gen(\dih(A))=
\begin{dcases}
*0          & \text{if } n \equiv_4 0 \\
*1          & \text{if } n \equiv_4 2 \\
*3          & \text{if } n \equiv_2 1. \\
\end{dcases}
\]
\end{proposition}

\begin{proposition}\label{prop:ManyGenerated}
If $A$ is an abelian group such that $\spr(A) \geq 3$, then $\gen(\dih(A))=*0$.
\end{proposition}

\begin{proof}
Since $\spr(A) \geq 3$, $\spr(\dih(A)) \geq 4$. The result follows by Corollary~\ref{cor:dgt4} because $\dih(A)$ is even.
\end{proof}

We now focus on the case when $\spr(A)=2$.  We start by considering the case when $A$ is even.
Figures~\ref{fig:genevenpow1}(b) and \ref{fig:genevenpow23}(a) show our conjectured two possible simplified structure diagrams for this situation.

\begin{proposition} \label{prop:Even2Generated}
If $A$ is an abelian group of even order such that $\spr(A)=2$, then $\gen(\dih(A))=*0$.
\end{proposition}

\begin{figure}
\includegraphics{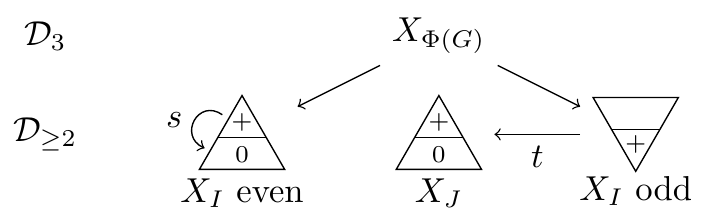}
\caption{\label{fig:Even2Generated}Visual representation for the proof of Proposition~\ref{prop:Even2Generated}.}
\end{figure}

\begin{proof}
Because $\spr(A)=2$, we know that $\spr(\dih(A))=3$ by Proposition~\ref{prop:NumberOfGenerators}, and so the initial position $\emptyset$ has deficiency 3.
We show that the second player can create a position with nim-number 0 in the second turn, no matter what the first player chooses in the first turn. This will show that the second player has a winning strategy and
the nim-number of the game is 0. 

Suppose the first player chooses $g \in \dih(A)$, and let $I=\lceil g \rceil$.  Then $X_I$ is in $\mathcal{D}_2$ or $\mathcal{D}_3$ by Proposition~\ref{prop:setoptions}.  
We have provided a visualization in Figure~\ref{fig:Even2Generated}.

If $X_{I}$ is in $\mathcal{E}_2$ or $\mathcal{E}_3$, then the second player picks any element $s\in I\setminus \{g\}$ to move to a position with nim-number $0$ by Proposition~\ref{prop:GENeventypes} since $\type(X_{I})$ is either $(0,0,1)$ or $(0,0,2)$.   

Now assume that $X_I$ is in $\mathcal{O}_2$ or $\mathcal{O}_3$.  Because every element in $\dih(A)\setminus A$ is an involution, we conclude that $I$ is an odd subgroup of $A$ and $g$ has odd order.   
The second player then selects an involution $t \in A$, and we consider a new structure class $X_J$ for $J=\lceil I,t \rceil$.  
Because $\gcd(|g|,|t|)=1$ and $A$ is abelian, $\langle g,t \rangle = \langle gt \rangle$ and hence $J=\lceil gt \rceil$.  
Therefore, $X_J$ is in $\mathcal{E}_2$  or $\mathcal{E}_3$.  Since $\type(X_{J})$ is either $(0,0,1)$ or $(0,0,2)$ by Proposition~\ref{prop:GENeventypes}, the second player moves to a position with nim-number $0$ by choosing $t$.
\end{proof}

We now consider the case when $A$ is odd.

\begin{lemma}\label{lem:GENoddsmalleroddoptions}
Let $G=\dih(A)$ where $A$ is an abelian group of odd order and $\spr(A)=2$.  If $X_I \in \mathcal{O}_m$ with $m \in \{2,3\}$, then $X_I$ has an option in $\mathcal{O}_{m-1}$. 
\end{lemma}

\begin{proof}
Recall that $C_2=\langle x\rangle$.
Since $X_I \in \mathcal{O}_m$, there are distinct $g_1,\ldots,g_m$ such that $\langle I, g_1, \ldots, g_m \rangle = G$.  Additionally, $I \leq A$ because $I$ is odd.  
If $g_i$ has odd order for some $i$, 
then $g_i \in A$ and $X_{\lceil I, g_i \rceil} \in \mathcal{O}_{m-1}$.

Now assume that all $g_1,\ldots,g_m$ have even order.  
Since $g_1$ and $g_2$ are involutions, there are $a, b \in A$ such that $g_1=xa$ and $g_2 = xb$.  
Then $\langle g_1, g_2 \rangle = \langle xa, xb \rangle = \langle a^{-1}b, xb \rangle$ if $m=2$ and $\langle g_1,g_2,g_3 \rangle =\langle a^{-1}b,g_2,g_3\rangle$ if $m=3$ 
by Lemma~\ref{lem:twoinvolutionstransform}, so $X_{\lceil I, a^{-1}b \rceil} \in \mathcal{O}_{m-1}$.  
\end{proof}

\begin{lemma}\label{lem:GENoddevenoptions}
Let $G=\dih(A)$ where $A$ is an abelian group of odd order such that $\spr(A)=2$.  If $X_I \in \mathcal{O}_m$ with $m \in \{1,2,3\}$, 
then $X_I$ has an option in $\mathcal{E}_{m-1}$ but not in $\mathcal{E}_{m}$. 
\end{lemma}

\begin{proof}
Since $X_I$ is odd, $I$ is a subgroup of $A$.
So $X_{\lceil I,h\rceil}$ is an option of $X_I$ for any $h\in G\setminus A$, and this option is in $\mathcal{E}_{m-1}$ or $\mathcal{E}_{m}$.
It remains to show that $X_I$ has no option in $\mathcal{E}_{m}$.

We argue by contradiction.  Assume that $X_I \in \mathcal{O}_m$ has an option $X_J \in \mathcal{E}_m$ and $m$ is as small as possible with this property.
Let $J=\lceil I,y \rceil$ for some $y\in G\setminus A$. 

We first consider the case where $m=1$ and $X_I \in \mathcal{O}_1$. 
Then $X_{\lceil I, g \rceil} \in \mathcal{D}_0$ for some $g \in G\setminus A$.  
Then $G=\langle I,g \rangle$, so $I=A$ by Lemma~\ref{lem:DihASubgroups}.
Thus $X_J=X_{\lceil A,y \rceil}=X_G \in \mathcal{E}_0$, which is a contradiction. 

Now we consider the case when $m\ge 2$, as shown in Figure \ref{fig:oddevenoptions}.
By Lemma \ref{lem:GENoddsmalleroddoptions}, there is a $g \in G$ such that $X_{\lceil I, g\rceil} \in \mathcal{O}_{m-1}$.  
Since $m$ is minimal, $X_{\lceil I, g\rceil}$ has no option in $\mathcal{E}_{m-1}$, and so  $X_{\lceil I, g, y\rceil} \in \mathcal{E}_{m-2}$.
This is a contradiction, since $X_{\lceil I, g, y\rceil}$ is an option of $X_J$ and $X_J$ does not have an option in $\mathcal{E}_{m-2}$ by
Proposition \ref{prop:setoptions}.

\begin{figure}
\includegraphics{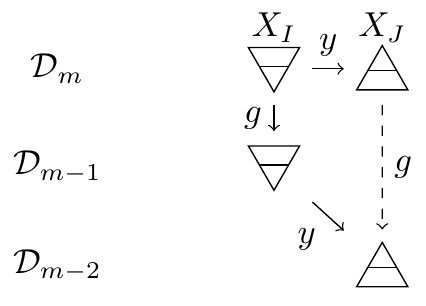}
\caption{\label{fig:oddevenoptions}The $m\ge 2$ case of the proof of Lemma~\ref{lem:GENoddevenoptions}.}
\end{figure}
\end{proof}

\begin{lemma}\label{lem:oddGENoddtypes}
If $G=\dih(A)$ where $A$ is an abelian group of odd order and $\spr(A)=2$, then 
\[
   \type(X_I) = 
    \begin{dcases}
       (1,2,1) & \text{if } X_I \in \mathcal{O}_1\\
       (1,3,0) & \text{if } X_I \in \mathcal{O}_2\\
       (1,3,1) & \text{if } X_I \in \mathcal{O}_{3}.
    \end{dcases}
\] 
\end{lemma}

\begin{proof}
We will use structural induction on the structure classes. Proposition \ref{prop:setoptions} and Lemmas \ref{lem:GENoddsmalleroddoptions} 
and \ref{lem:GENoddevenoptions} determine the possible options of our structure classes. 
Proposition~\ref{prop:GENeventypes} and induction yields the types of these options.  

If $X_I\in\mathcal{O}_1$, then $\opt(X_I)\cap \mathcal{E}_0\ne \emptyset$ and $\opt(X_I)\subseteq\mathcal{E}_0\cup \mathcal{O}_1$.  
So $\type(X_I)=(1,2,1)$ since
\[
\otype(X_I)=
\begin{dcases}
\{(0,0,0)\}          & \text{if } \opt(X_I)\cap \mathcal{O}_1 = \emptyset \\
\{(0,0,0),(1,2,1)\}  & \text{otherwise}.
\end{dcases}
\]

If $X_I\in\mathcal{O}_2$, then $\opt(X_I) \cap \mathcal{E}_1 \ne \emptyset\ne \opt(X_I)\cap \mathcal{O}_1$ 
and $\opt(X_I)\subseteq\mathcal{E}_1\cup\mathcal{O}_1\cup \mathcal{O}_2$. 
So $\type(X_I)=(1,3,0)$ since
\[
\otype(X_I)=
\begin{dcases}
\{(0,1,2),(1,2,1)\}          & \text{if } \opt(X_I) \cap \mathcal{O}_2 = \emptyset \\
\{(0,1,2),(1,2,1),(1,3,0)\}  & \text{otherwise}.
\end{dcases}
\]

If $X_I\in\mathcal{O}_3$, then $\opt(X_I)\cap \mathcal{E}_2 \ne \emptyset\ne \opt(X_I)\cap \mathcal{O}_2$ 
and $\opt(X_I)\subseteq\mathcal{E}_2\cup\mathcal{O}_2\cup \mathcal{O}_3$. 
So $\type(X_I)=(1,3,1)$ since
\[
\otype(X_I)=
\begin{dcases}
\{(0,0,2), (1,3,0)\}         & \text{if } \opt(X_I) \cap \mathcal{O}_3 = \emptyset \\
\{(0,0,2),(1,3,0),(1,3,1)\}  & \text{otherwise}.
\end{dcases}
\]
\end{proof}

Figure~\ref{fig:genodd}(b) shows our conjectured simplified structure diagram that corresponds to the next result.

\begin{proposition}\label{prop:Odd2Generated}
If $A$ is an abelian group of odd order such that $d(A)=2$, then $\gen(\dih(A))=*3$.
\end{proposition}

\begin{proof}
By Proposition~\ref{prop:frattinilocation}, we have $X_{\Phi(G)} \in \mathcal{D}_3$ since $\spr(\dih(A))=3$.   
By Lemma~\ref{prop:FrattiniInclusion}, we have $\Phi(G)=\Phi(A) \leq A$, so $\Phi(G)$ is odd.  
Therefore,  $X_{\Phi(G)} \in  \mathcal{O}_3$, so $\type(X_{\Phi(G)})=(1,3,1)$ by Lemma~\ref{lem:oddGENoddtypes}.
The second component is $3$, so we conclude that $\gen(\dih(A))=*3$. 
\end{proof}

We conclude with our main result, which summarizes Propositions~\ref{prop:DihedralResults},~\ref{prop:ManyGenerated},~\ref{prop:Even2Generated}, and~\ref{prop:Odd2Generated}. 

\begin{theorem}
If $A$ is an abelian group, then 
\[
\gen(\dih(A))=
\begin{dcases}
*1 & \text{if } \spr(A)=1, |A|\equiv_4 2\\
*3 & \text{if } 1\le\spr(A)\le2, A \text{ odd}\\
*0 & \text{otherwise}. 
\end{dcases}
\]

\end{theorem}


\section{Further Questions}


\begin{figure}
\begin{tabular}{ccccccc}
\includegraphics[scale=0.72]{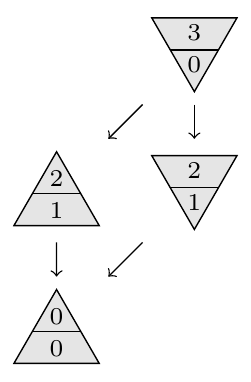} &  & \includegraphics[scale=0.72]{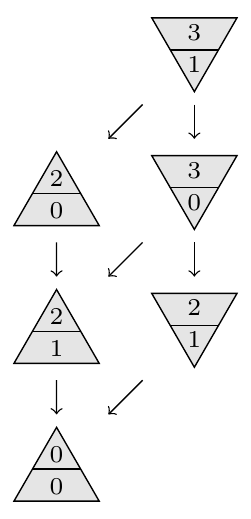} &  & \includegraphics[scale=0.72]{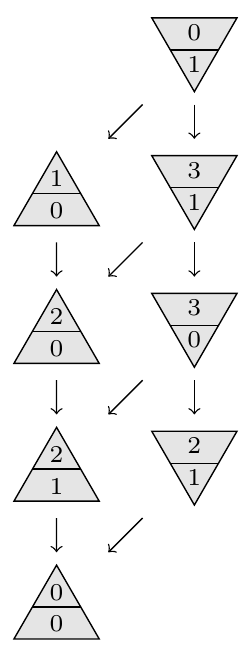} &  & \includegraphics[scale=0.72]{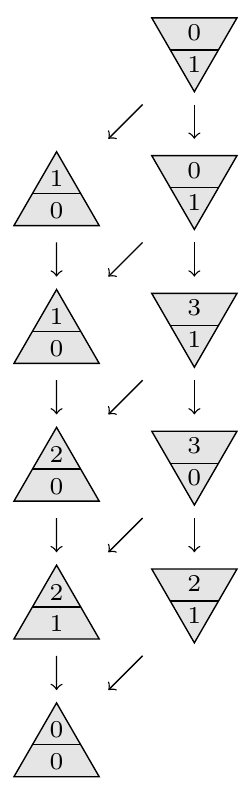}\tabularnewline
(a) $\spr(A)=1$ &  & (b) $\spr(A)=2$ &  & (c) $\spr(A)=3$ &  & (d) $\spr(A)\ge4$\tabularnewline
\end{tabular}

\protect\caption{\label{fig:genodd}Conjectured simplified structure diagrams for $\protect\gen(\protect\dih(A))$
with odd $A$.}
\end{figure}

\begin{figure}
\begin{tabular}{ccccccccc}
\includegraphics[scale=0.72]{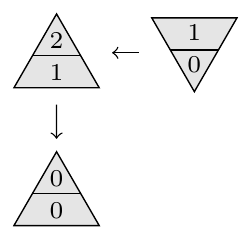} &  & \includegraphics[scale=0.72]{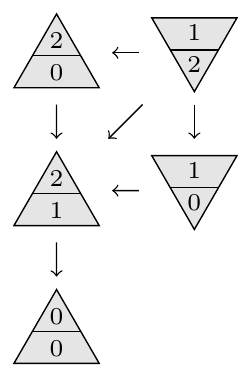} &  & \includegraphics[scale=0.72]{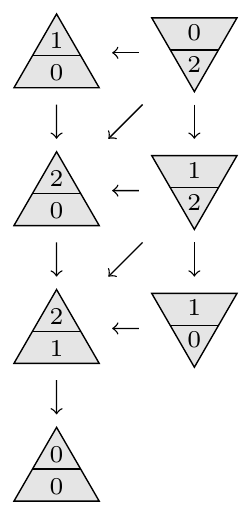} &  & \includegraphics[scale=0.72]{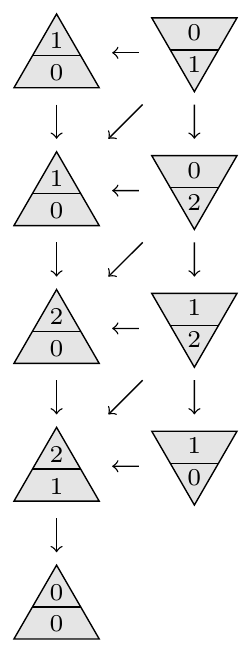} &  & \includegraphics[scale=0.72]{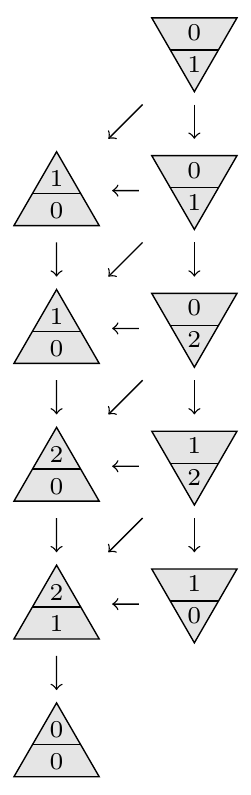}\tabularnewline
(a) $\spr(P)=0$ &  & (b) $\spr(P)=1$  &  & (c) $\spr(P)=2$ &  & (d) $\spr(P)=3$ &  & (d) $\spr(P)\ge4$\tabularnewline
\end{tabular}

\protect\caption{\label{fig:genevenpow1}Conjectured simplified structure diagrams for $\protect\gen(\protect\dih(A))$
with $A=\mathbb{Z}_{2}\times P$ and $P$ odd.}
\end{figure}

\begin{figure}
\begin{tabular}{cccccccccc}
\includegraphics[scale=0.72]{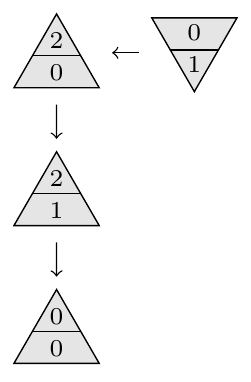} &  & \includegraphics[scale=0.72]{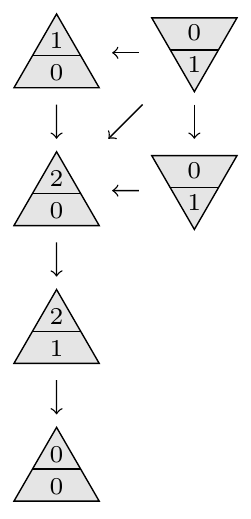} &  & \includegraphics[scale=0.72]{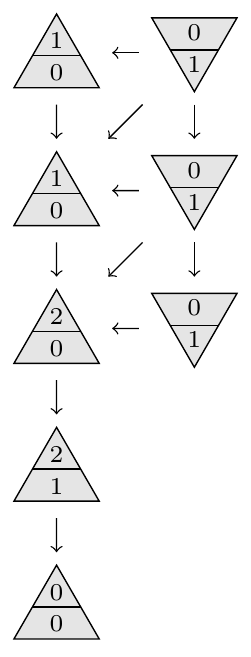} & &
\includegraphics[scale=0.72]{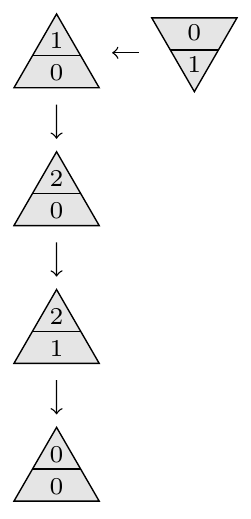} &  & \includegraphics[scale=0.72]{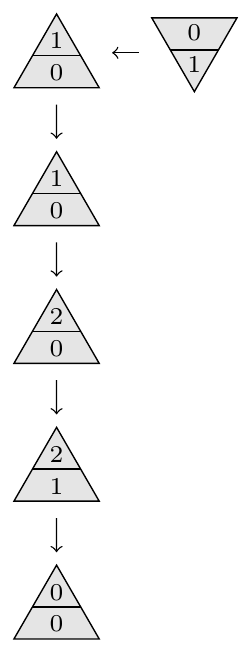}\tabularnewline
(a) $k=2$          &  & (b) $k=2$  &  & (c) $k=2$ & & (d) $k=3$ &  & (e) $k\ge4$\tabularnewline
$0\le\spr(P)\le1$     &  &     $\spr(P)=2$ &  &     $\spr(P)\ge3$ & \phantom{$\int_a^b$} &  &  &            \tabularnewline
\end{tabular}
\caption{\label{fig:genevenpow23}Conjectured simplified structure diagrams for $\protect\gen(\protect\dih(A))$
with $A=\mathbb{Z}_{2}^{k}\times P$ and $P$ odd.}
\end{figure}

We outline a few open problems.

\begin{enumerate}
\item The simplified structure diagrams in Figures \ref{fig:genodd}--\ref{fig:genevenpow23} were found via computer calculations by
considering several examples of $\gen(\dih(A))$. We conjecture that these figures cover all possible structure diagrams.  

\item We believe that the key new insight of this paper is the use of deficiency. Can we use this tool to handle other semidirect products?

\item Can we extend the techniques in this paper to give a complete description of the nim-numbers of $\gen(G)$ using covering properties as was done in~\cite{BeneshErnstSiebenDNG} for $\dng(G)$?
It is likely that $d(G)$ plays a role in this characterization.
\end{enumerate}

\bibliographystyle{plain}
\bibliography{game}

\end{document}